\documentclass[a4paper,reqno,11pt]{amsart}
\usepackage{amssymb,amsmath,amsthm}
\usepackage{amsfonts}
\usepackage{latexsym}
\usepackage{graphicx}
\usepackage[english]{babel}

\usepackage[usenames]{color}
\usepackage{amssymb}
\usepackage{graphicx}
\usepackage{amscd}

\usepackage[colorlinks=true,
linkcolor=webgreen, filecolor=webbrown, citecolor=webgreen]{hyperref}

\definecolor{webgreen}{rgb}{0,.5,0}
\definecolor{webbrown}{rgb}{.6,0,0}

\usepackage{color}
\usepackage{float}
\usepackage{epsf}

\begin{document}

\title[]{Determinants of grids, tori, cylinders\\ and M\"{o}bius ladders}

\author{Khodakhast\ Bibak$^1$}

\author{Roberto\ Tauraso$^2$}

\footnotetext[1]{Department of Combinatorics and Optimization,
University of Waterloo, Waterloo, Ontario, Canada N2L 3G1\\
Email: {\tt kbibak@uwaterloo.ca}}

\footnotetext[2]{Dipartimento di Matematica,
Universit\`a di Roma ``Tor Vergata'', 00133 Roma, Italy\\
Email: {\tt tauraso@mat.uniroma2.it}}

\maketitle
\begin{abstract}
Recently, Bie\~{n} [A. Bie\~{n}, The problem of singularity for
planar grids, {\it Discrete Math.} 311 (2011), 921--931] obtained
a recursive formula for the determinant of a grid. Also, recently,
Pragel [D. Pragel, Determinants of box products of paths, {\it
Discrete Math.} 312 (2012), 1844--1847], independently, obtained
an explicit formula for this determinant. In this paper, we give a
short proof for this problem. Furthermore, applying the same
technique, we get explicit formulas for the determinant of a
torus, a cylinder, and a M\"{o}bius ladder.
\end{abstract}

\newtheorem{theorem}{Theorem}
\newtheorem{corollary}[theorem]{Corollary}
\newtheorem{lemma}[theorem]{Lemma}
\newtheorem{proposition}[theorem]{Proposition}
\newtheorem{conjecture}[theorem]{Conjecture}
\newtheorem{defin}[theorem]{Definition}
\newenvironment{definition}{\begin{defin}\normalfont\quad}{\end{defin}}
\newtheorem{examp}[theorem]{Example}
\newenvironment{example}{\begin{examp}\normalfont\quad}{\end{examp}}
\newtheorem{rema}[theorem]{Remark}
\newenvironment{remark}{\begin{rema}\normalfont\quad}{\end{rema}}

\newcommand{\len}{\mbox{len}}
\newcommand{\bal}[1]{\begin{align*}#1\end{align*}}
\newcommand{\f}[2]{\displaystyle \frac{#1}{#2}}
\newcommand{\bt}{\begin{thm}}
\newcommand{\et}{\end{thm}}
\newcommand{\bp}{\begin{proof}}
\newcommand{\ep}{\end{proof}}
\newcommand{\bprop}{\begin{prop}}
\newcommand{\eprop}{\end{prop}}
\newcommand{\bl}{\begin{lemma}}
\newcommand{\el}{\end{lemma}}
\newcommand{\bc}{\begin{corollary}}
\newcommand{\ec}{\end{corollary}}
\newcommand{\Z}{\mathbb{Z}}
\newcommand{\be}{\begin{enumerate}}
\newcommand{\ee}{\end{enumerate}}

\section{Introduction and results}

We denote by $A(G)$ the adjacency matrix of a graph $G$. A path
(cycle) on $n$ vertices is denoted by $P_n$ (resp., $C_n$). Given
two graphs $G_1=(V_1,E_1)$ and $G_2=(V_2,E_2)$, their {\it
Cartesian product} $G_1\; \Box \;G_2$ is the graph with vertex set
$V_1\times V_2$ and edge set
$$\Big\{\big((u,v),(u',v)\big):(u,u')\in E_1, v\in V_2\Big\}\bigcup \Big\{\big((u,v),(u,v')\big): u \in V_1, (v,v')\in E_2\Big\}.$$
Cartesian product produces many important classes of graphs. For
example, a {\it grid} (also called {\it mesh}) is the Cartesian
product of two paths, a {\it torus} (also called {\it toroidal
grid} or {\it toroidal mesh}) is the Cartesian product of two
cycles, and a {\it cylinder} is the Cartesian product of a path
and a cycle. One can generalize these definitions to more than two
paths or cycles. These classes of graphs are widely used computer
architectures (e.g., grids are widely used in multiprocessor VLSI
systems) \cite{HL}.

The {\it nullity} of a graph $G$ of order $n$, denoted by
$\eta(G)$, is the multiplicity of 0 in the spectrum of $G$.
Clearly, $\eta(G) = n - r(A(G))$, where $r(A(G))$ is the rank of
$A(G)$. The nullity of a graph is closely related to the minimum
rank problem of a family of matrices associated with a graph (see,
e.g., \cite{FAHO} and the references therein). Nullity of a
(molecular) graph (specifically, determining whether it is
positive or zero) has also important applications in quantum
chemistry and H\"{u}ckel molecular orbital (HMO) theory (see,
e.g., \cite{GUBO} and the references therein). A famous problem,
posed by Collatz and Sinogowitz in 1957 \cite{COSI}, asks to
characterize all graphs with positive nullity. Clearly, $\det
A(G)=0$ if and only if $\eta(G)>0.$ So, examining the determinant
of a graph is a way to attack this problem. But there seems to be
little published on calculating the determinant of various classes
of graphs.

Recently, Bie\~{n} \cite{BIE} obtained a recursive formula for the
determinant of a grid. Also, recently, Pragel \cite{PRA},
independently, obtained an explicit formula for this determinant
(see \eqref{T1} below). Here, using trigonometric identities, we
give a short proof for this problem. Furthermore, applying the
same technique, we get explicit formulas for the determinant of a
torus, and a cylinder.

\begin{theorem} \label{thm:detcar}
Let $m>1$ and $n>1$ be integers. Then
\begin{align}\label{T1}
 \det A(P_{m-1}\; \Box \; P_{n-1})&=
  \begin{cases}
   (-1)^{\frac{(m-1)(n-1)}{2}} & \text{if}\; \gcd(m,n) = 1;\\
    0 & \text{otherwise}.
  \end{cases}\\\label{T2}
  \det A(C_{m}\; \Box \; C_{n})&=
  \begin{cases}
   4^{\gcd(m,n)}\qquad & \text{if $m$ and $n$ are odd};\\
    0\qquad & \text{otherwise}.
  \end{cases}\\\label{T3}
  \det A(P_{m-1}\; \Box \; C_{n})&=
  \begin{cases}
   m & \text{if $n$ is odd and $\gcd(m,n)=1$};\\
   (-1)^{m-1}m^2 & \text{if $n$ is even and $\gcd(m,n/2)=1$};\\
   0 & \text{otherwise}.
  \end{cases}
\end{align}
\end{theorem}

Note that for $m=2$, \eqref{T1} and \eqref{T3} give the following
well-known determinants.
\begin{align*}
 \det A(P_{n-1})=
  \begin{cases}
   (-1)^{\frac{n-1}{2}} & \text{if $n$ is odd};\\
    0 & \text{otherwise}.
  \end{cases}
  \quad\mbox{and}\quad
  \det A(C_{n})=
  \begin{cases}
   2 & \text{if $n$ is odd};\\
   -4& \text{if $n\equiv 2 \pmod{4}$};\\
   0 & \text{otherwise}.
  \end{cases}
\end{align*}
Our proof techniques or its modifications may be useful in other
situations with similar flavor (see, e.g., \cite{BISH}). For
example, let us consider the {\it M\"{o}bius ladder} $M_{2n}$, the
graph on $2n$ vertices whose edge set is the union of the edge set
of $C_{2n}$ and $\{(v_i,v_{n+i}):i=1,\dots, n\}$. We prove that

\begin{theorem} \label{thm:mob}
Let $n>1$ be an integer. Then
\begin{align}\label{T4}
 \det A(M_{2n})&=
  \begin{cases}
   -3 & \text{if $n\equiv \pm 2 \pmod{6}$;}\\
   -9 & \text{if $n\equiv \pm 1 \pmod{6}$;}\\
    0 & \text{otherwise}.
  \end{cases}
\end{align}
\end{theorem}


\section{Techniques and Proofs}

The starting point of our calculations is the following well-known
theorem which gives the eigenvalues of the Cartesian product of
two graphs (see, e.g., \cite[p. 587]{ROS}).

\begin{theorem} \label{thm:eigcar}
Let $G_1$ be a graph of order $m$, and $G_2$ be a graph of order
$n$. If the eigenvalues of $A(G_1)$ and $A(G_2)$ are,
respectively, $\lambda_{1}(G_1),\ldots,\lambda_{m}(G_1)$ and
$\lambda_{1}(G_2),\ldots,\lambda_{n}(G_2)$, then the eigenvalues
of $A(G_1\; \Box \;G_2)$ are precisely the numbers $
\lambda_{i}(G_1)+\lambda_{j}(G_2),$ for $i=1,2,\ldots,m$ and
$j=1,2,\ldots,n$.
\end{theorem}

We also need the following trigonometric identities, which might
be of independent interest.

\begin{lemma} \label{lem:triiden}
Let $n$ be a positive integer and let $a\in\mathbb{Z}$ such that $\gcd(a,n)=1$. Then for any real number $x$,
\begin{equation}\label{eq1}
\sin(nx)=2^{n-1}(-1)^{\frac{(a-1)(n-1)}{2}}\prod_{j=0}^{n-1}\sin\left(x+{a j\pi \over n}\right).
\end{equation}
Moreover,
\begin{equation}\label{eq2}
\prod_{j=1}^{n-1}\sin\left({a j\pi \over n}\right)=(-1)^{\frac{(a-1)(n-1)}{2}}\cdot{n\over 2^{n-1}}
\end{equation}
and
\begin{equation}\label{eq3}
\prod_{j=1}^{n-1}\cos\left({a j\pi \over n}\right)=
  \begin{cases}
   (-1)^{\frac{a(n-1)}{2}}\cdot{1\over 2^{n-1}} & \mbox{if $n$ is odd};\\
    0 & \mbox{otherwise}.
  \end{cases}
\end{equation}
\end{lemma}
\begin{proof} Let $\omega=e^{\pi aI/n}$, where $I=\sqrt{-1}$. Then, since
$\{\omega^{-2j}:j=0,\dots,n-1\}$ are all the $n$-th roots of unity, it follows that
$$\prod_{j=0}^{n-1}(z-\omega^{-2j})=z^n-1.$$
Hence,
\begin{align*}
\prod_{j=0}^{n-1}\sin\left(x+{a j\pi \over n}\right)&=
\prod_{j=0}^{n-1}\frac{e^{Ix}\omega^{j}-e^{-Ix}\omega^{-j}}{2I}\\
&=\frac{e^{-nIx}\omega^{n(n-1)/2}}{(2I)^n}\prod_{j=0}^{n-1}(e^{2Ix}-\omega^{-2j})\\
&=\frac{I^{(a-1)(n-1)}}{2^{n-1}}\cdot\frac{e^{nIx}-e^{-nIx}}{2I}\\
&=\frac{(-1)^{\frac{(a-1)(n-1)}{2}}}{2^{n-1}}\cdot\sin(nx).
\end{align*}
Moreover,
$$\prod_{j=1}^{n-1}\sin\left({a j\pi \over n}\right)
={(-1)^{\frac{(a-1)(n-1)}{2}}\over 2^{n-1}}\lim_{x\to
0}\frac{\sin(nx)}{\sin x} =(-1)^{\frac{(a-1)(n-1)}{2}}\cdot{n\over
2^{n-1}}.$$ Finally,
$$\prod_{j=1}^{n-1}\cos\left({a j\pi \over n}\right)
=(-1)^{n-1}\prod_{j=1}^{n-1}\sin\left(-{\pi\over 2}+{a j\pi \over
n}\right)
=(-1)^{n-1}\sin(n\pi/2)\cdot{(-1)^{\frac{(a-1)(n-1)}{2}}\over
2^{n-1}},$$ which easily yields the required formula.
\end{proof}

Now, we are ready to prove Theorem \ref{thm:detcar}.

\begin{proof}[Proof of Theorem \ref{thm:detcar}]
It is known (see, e.g., \cite[p. 588]{ROS}) that the eigenvalues
of $A(P_{m-1})$ and $A(C_{n})$ are, respectively,
$$\left\{2\cos\left(\frac{i\pi}{m}\right)\;:\; 1\leq i\leq m-1\right\}\quad\mbox{and}\quad
\left\{2\cos\left(\frac{2j\pi}{n}\right)\;:\; 1\leq j\leq n\right\}.$$

The proof is done by a direct combination of
Theorem~\ref{thm:eigcar} and Lemma~\ref{lem:triiden}.

We start with \eqref{T1}. Using the identity $\cos(a+b)+\cos(a-b)=2\cos(a)\cos(b)$,
\begin{align*}
\det A(P_{m-1}\; \Box \; P_{n-1})&=\prod_{i=1}^{m-1}\prod_{j=1}^{n-1}\Big(2\cos\left(\frac{i\pi}{m}\right)+2\cos\left(\frac{j\pi}{n}\right)\Big)\\
&=2^{(m-1)(n-1)}\prod_{i=1}^{m-1}\prod_{j=1}^{n-1}2\cos\Big(\frac{i\pi}{2m}+\frac{j\pi}{2n}\Big)\cos\Big(\frac{i\pi}{2m}-\frac{j\pi}{2n}\Big)\\
&=2^{(m-1)(n-1)}\prod_{i=1}^{m-1}\prod_{j=1}^{n-1}2\cos\Big(\frac{i\pi}{2m}+\frac{j\pi}{2n}\Big)\cos\Big(\frac{i\pi}{2m}-\frac{(n-j)\pi}{2n}\Big)\\
&=2^{(m-1)(n-1)}\prod_{i=1}^{m-1}\prod_{j=1}^{n-1}2\cos\Big(\frac{i\pi}{2m}+\frac{j\pi}{2n}\Big)\sin\Big(\frac{i\pi}{2m}+\frac{j\pi}{2n}\Big)\\
&=2^{(m-1)(n-1)}\prod_{i=1}^{m-1}\prod_{j=1}^{n-1}\sin\Big(\frac{i\pi}{m}+\frac{j\pi}{n}\Big)\\
&=2^{(m-1)(n-1)}\prod_{i=1}^{m-1}
\frac{\sin\left(\frac{ni\pi}{m}\right)}{2^{n-1}\sin\left(\frac{i\pi}{m}\right)}
=\frac{\prod_{i=1}^{m-1}\sin\left(\frac{ni\pi}{m}\right)}{\prod_{i=1}^{m-1}\sin\left(\frac{i\pi}{m}\right)},
\end{align*}
where in the last but one step we have used the identity \eqref{eq1}.
Clearly, if $\gcd(m,n)\not=1$ then
$\prod_{i=1}^{m-1}\sin\left(\frac{ni\pi}{m}\right)=0$, otherwise we use \eqref{eq2}.

Now, we show \eqref{T2}. In the case that $m$ or $n$ is even the
proof is straightforward because one of the eigenvalues of
$A(C_{m}\; \Box \; C_{n})$ is zero. Assume that $m$ and $n$ are
odd and let $d=\gcd(m,n)$, with $m'=m/d$, $n'=n/d$.
\begin{align*}
\det A(C_{m}\; \Box \; C_{n})&=\prod_{i=1}^{m}\prod_{j=1}^{n}\left(2\cos\left(\frac{2i\pi}{m}\right)+2\cos\left(\frac{2j\pi}{n}\right)\right)\\
&=4^{mn}\prod_{i=0}^{m-1}\prod_{j=0}^{n-1}\cos\Big(\frac{i\pi}{m}+\frac{j\pi}{n}\Big)\cos\Big(\frac{i\pi}{m}-\frac{j\pi}{n}\Big)\\
&=4^{mn}\Bigg(\prod_{i=0}^{m-1}\prod_{j=0}^{n-1}
\cos\Big(\frac{i\pi}{m}+\frac{j\pi}{n}\Big)\cos\Big(\frac{i\pi}{m}-\frac{(n-j)\pi}{n}\Big)\Bigg)\\
&=4^{mn}\Bigg(\prod_{i=0}^{m-1}\prod_{j=0}^{n-1}
\cos\Big(\frac{i\pi}{m}+\frac{j\pi}{n}\Big)\Bigg)^2\\
&=4^{mn}\Bigg(\prod_{i=0}^{m-1}\prod_{j=0}^{n-1}
\sin\Big(-\frac{\pi}{2}+\frac{i\pi}{m}+\frac{j\pi}{n}\Big)\Bigg)^2\\
&=4^{m}\Bigg(\prod_{i=0}^{m-1}
\sin\left(n\left(-\frac{\pi}{2}+\frac{i\pi}{m}\right)\right)\Bigg)^2\\
&=4^{m}\Bigg(\prod_{i=0}^{m-1} \cos\left(\frac{ni\pi}{m}\right)\Bigg)^2
=4^{m}\Bigg(\prod_{i=1}^{m' d} \cos\left(\frac{n'i\pi}{m'}\right)\Bigg)^2=4^{m}\Big(\frac{1}{4^{m'-1}}\Big)^d=4^d,
\end{align*}
where in the last but one step we have used \eqref{eq3}.

Finally, we prove \eqref{T3}.
\begin{align*}
\det A(P_{m-1}\; \Box \; C_{n})&=\prod_{i=1}^{m-1}\prod_{j=1}^{n}\left(2\cos\left(\frac{i\pi}{m}\right)+2\cos\left(\frac{2j\pi}{n}\right)\right)\\
&=4^{(m-1)n}\prod_{i=1}^{m-1}\prod_{j=0}^{n-1}
\cos\Big(\frac{i\pi}{2m}+\frac{j\pi}{n}\Big)\cos\Big(\frac{i\pi}{2m}-\frac{j\pi}{n}\Big)\\
&=4^{(m-1)n}\prod_{i=1}^{m-1}\prod_{j=0}^{n-1}
\cos\Big(\frac{(m-i)\pi}{2m}+\frac{j\pi}{n}\Big)\cos\Big(\frac{(m-i)\pi}{2m}-\frac{j\pi}{n}\Big)\\
&=(-4)^{(m-1)n}\prod_{i=1}^{m-1}\prod_{j=0}^{n-1}
\sin\Big(-\frac{i\pi}{2m}+\frac{j\pi}{n}\Big)\sin\Big(\frac{i\pi}{2m}+\frac{j\pi}{n}\Big)\\
&=(-4)^{(m-1)n}\prod_{i=1}^{m-1}\frac{1}{4^{n-1}}
\sin\Big(-\frac{ni\pi}{2m}\Big)\sin\Big(\frac{ni\pi}{2m}\Big)\\
&=(-1)^{(m-1)(n-1)}4^{m-1}\Bigg(\prod_{i=1}^{m-1}
\sin\Big(\frac{ni\pi}{2m}\Big)\Bigg)^2.
\end{align*}
If $n$ is even and $\gcd(m,n')=1$ where $n'=n/2$ then, by \eqref{eq2},
$$(-1)^{(m-1)(n-1)}4^{m-1}\Bigg(\prod_{i=1}^{m-1}
\sin\Big(\frac{ni\pi}{2m}\Big)\Bigg)^2=
(-1)^{(m-1)}4^{m-1}\Bigg(\prod_{i=1}^{m-1}\sin\Big(\frac{n'i\pi}{m}\Big)\Bigg)^2=
(-1)^{(m-1)}m^2.
$$
If $n$ is odd and $\gcd(m,n)=1$ then, $\gcd(2m,n)=1$ and by \eqref{eq2},
\begin{align*}
(-1)^{(m-1)(n-1)}4^{m-1}\Bigg(\prod_{i=1}^{m-1}
\sin\Big(\frac{ni\pi}{2m}\Big)\Bigg)^2
&=4^{m-1}\Bigg(\prod_{i=1}^{m-1}
\sin\Big(\frac{ni\pi}{2m}\Big)\Bigg)\Bigg(\prod_{i=m+1}^{2m-1}
\sin\Big(\frac{n(2m-i)\pi}{2m}\Big)\Bigg)\\
&=4^{m-1}\sin\Big(\frac{n\pi}{2}\Big)\prod_{i=1}^{2m-1}\sin\Big(\frac{ni\pi}{2m}\Big)=m
\end{align*}
It is easy to verify that the remaining cases yield zero.
\end{proof}

Now, we prove Theorem \ref{thm:mob}.

\begin{proof}[Proof of Theorem \ref{thm:mob}]

The eigenvalues of $A(M_{2n})$ are (see, e.g., \cite[p. 21]{BIG})
$$\left\{(-1)^j + 2\cos\left(\frac{j\pi}{n}\right)\;:\; 1\leq j\leq 2n\right\}.$$
Hence,
\begin{align*}
\det A(M_{2n})&=\prod_{j=1}^{2n}\left((-1)^j + 2\cos\left(\frac{j\pi}{n}\right)\right)\\
&=\prod_{j=0}^{2n-1}\left(2\cos\left(\frac{(3j+1)\pi}{3}\right) + 2\cos\left(\frac{j\pi}{n}\right)\right)\\
&=4^{2n}\prod_{j=0}^{2n-1}\left(\cos\left(\frac{(3j+1)\pi}{6}+\frac{j\pi}{2n}\right)\cos\left(\frac{(3j+1)\pi}{6}-\frac{j\pi}{2n}\right) \right)\\
&=4^{2n}\prod_{j=0}^{2n-1}\sin\left(\frac{\pi}{3}-\frac{(n+1)j\pi}{2n}\right)\prod_{j=0}^{2n-1}\sin\left(\frac{\pi}{3}-\frac{(n-1)j\pi}{2n}\right).
\end{align*}
If $n$ is even then $\gcd(n+1,2n)=1$, and then by \eqref{eq1},
$$\det A(M_{2n})=-4\sin^2\left(\frac{2n\pi}{3}\right)=\begin{cases}
   -3 & \text{if $n\equiv \pm 2 \pmod{6}$;}\\
    0 & \text{if $n\equiv 0 \pmod{6}$.}
  \end{cases}$$
If $n$ is odd then $\gcd(n'+1,n)=1$, where $n'=(n-1)/2$, and then
by \eqref{eq1},
\begin{align*}
\det A(M_{2n})&=
-4^{2n}\left(\prod_{j=0}^{n-1}\sin\left(\frac{\pi}{3}-\frac{(n'+1)j\pi}{n}\right)\right)^4\\
&=-16\sin^4\left(\frac{n\pi}{3}\right)=\begin{cases}
   -9 & \text{if $n\equiv \pm 1 \pmod{6}$;}\\
    0 & \text{if $n\equiv 3 \pmod{6}$.}
  \end{cases}
\end{align*}
\end{proof}


\end{document}